\documentclass[11pt,reqno]{amsart}

\usepackage{amsmath,amsthm,amsfonts,epsfig,color}
\usepackage[latin1]{inputenc}
\usepackage{hyperref}

\newtheorem{theorem}{Theorem}[section]
\newtheorem{lemma}[theorem]{Lemma}
\newtheorem{proposition}[theorem]{Proposition}
\newtheorem{definition}[theorem]{Definition}
\newtheorem{remark}[theorem]{Remark}
\newtheorem{corollary}[theorem]{Corollary}

\newtheorem{problem}{Problem}

\renewcommand*{\div}{\ensuremath{\mathrm{div\,}}}

\def\deg{\textrm{deg}\,}

\def\eps{\varepsilon}
\def\dist{{\rm dist}}
\def\sym{{\mathrm{sym\,}}}

\def\S{\mathcal{S}}
\def\D{\mathcal{D}}
\def\R{\mathbb{R}}
\def\T{\mathbb{T}}

\def\N{\mathbb{N}}

\newcommand\e{{\varepsilon}}

\begin{document}

\title[The $h$-principle and fluid dynamics]{The $h$-principle and the equations of fluid dynamics}

\author{Camillo De Lellis}
\address{Institut f\"ur Mathematik, Universit\"at Z\"urich, 
CH-8057 Z\"urich}
\email{camillo.delellis@math.unizh.ch}

\author{L\'aszl\'o Sz\'ekelyhidi Jr.}
\address{Hausdorff Center for Mathematics, Universit\"at Bonn, 
D-53115 Bonn}
\email{laszlo.szekelyhidi@hcm.uni-bonn.de}

\begin{abstract}
In this note we survey some recent results for the
Euler equations in compressible and incompressible
fluid dynamics. The main point of all these
theorems is the surprising fact that a suitable 
variant of Gromov's $h$-principle holds in several
cases.
\end{abstract}

\maketitle

%%%%%%%%%%%%%%%%%%%%%%%%%%%%%%%
\section{Introduction}
%%%%%%%%%%%%%%%%%%%%%%%%%%%%%%%

The starting point of this note is the Cauchy problem for the incompressible Euler equations,
\begin{equation}\label{e:Cauchy}
\left\{
\begin{array}{l}
\partial_t v + {\rm div}_x v\otimes v + \nabla p = 0,\\ 
{\rm div}_x\, v = 0, \\
v (0, \cdot
) = v_0,
\end{array}\right.
\end{equation}
where the unknowns $v$ and $p$ are, respectively, a vectorfield and a scalar function
defined on $\R^n\times [0,T)$. These fundamental equations were derived over 250 years ago by Euler and since then have played a major role in fluid dynamics.
There are several oustanding open problems connected to \eqref{e:Cauchy} and an extensive literature
about them (see for instance the surveys \cite{BardTiti}, \cite{ConstSurvey}, \cite{Fefferman}). 
In three space dimensions
little is known about smooth solutions apart from classical short time existence and uniqueness (see \cite[p 422]{Lichtenstein})  and the celebrated blow-up criterion of Beale-Kato-Majda \cite{BeKaMaj} (further references for both results are
\cite{EbinMarsden}, \cite{Kato}, \cite{Koch} and also the books \cite{MajdaBook} and \cite{MarcPulv}).
On the other hand weak solutions are known to be badly behaved from the point of view of 
Hadamard's well-posedness theory: in the groundbreaking paper \cite{Scheffer93} V.~Scheffer proved the
existence of a nontrivial weak solution compactly supported in time. Nevertheless, weak solutions have 
been studied for
their expected relevance to the theory of turbulence (see \cite{Onsager}, \cite{ConstantinETiti}, \cite{Eyink},
\cite{Shnirelmandecrease}). 

In this survey we argue that the nonuniqueness of weak solutions for the incompressible Euler
equations should be viewed as a suitable variant of the $h$-principle. The original $h$-principle
of Gromov pertains to various problems in differential geometry, where one expects high flexibility
of the moduli space of solutions due to the underdetermined nature of the problem. It was not expected that the same principle and similar methods
could be applied to problems in mathematical physics (we quote Gromov's speech at the
Balzan Prize \cite{GromovBalzan}: {\em The class
of infinitesimal laws subjugated by the homotopy principle is wide, but it does not
include most partial differential equations (expressing infinitesimal laws) of physics
with a few exceptions in favour of this principle leading to unexpected solutions. In
fact, the presence of the h-principle would invalidate the very idea of a physical law as
it yields very limited global information effected by the infinitesimal data}. See also the
introduction in the book \cite{Eliashberg}). 

As pointed out in the important paper \cite{MS99} by S.~M\"uller and V.~\v{S}verak
the existence theory for differential inclusions (see also \cite{BressanFlores,Cellina,DacorognaMarcellini97,Bernd,KMS02}) has a close
relation to the $h$-principle. In particular the method of convex integration, introduced by Gromov and extended
by M\"uller and \v{S}verak to Lipschitz mappings, provides a very
powerful tool to construct solutions. 
In the paper \cite{DS1} these tools were suitably modifed and used for the first time to
explain Scheffer's nonuniqueness theorem. It was also noticed immediately that this
approach allows to go way beyond the result of Scheffer. Indeed it has lead to new
developments for several equations in fluid dynamics. The goals of this note is to survey
these results, list some open questions and point our some new directions. The most
exciting ones address the question of whether these methods might give us some better
understanding of the theory of turbulence.

%%%%%%%%%%%%%%%%%%%%%%%%%%%%%%%
\section{Weak solutions}
%%%%%%%%%%%%%%%%%%%%%%%%%%%%%%%

We start with a survey of the non-uniqueness theorems for weak solutions of \eqref{e:Cauchy}.
By a weak solution we mean, as usual, an $L^2_{loc}$ vectorfield which solves the equations in the sense of distributions.
\begin{definition}\label{d:weaksolution}
A vectorfield $v\in L^2_{loc}(\R^n\times (0,T))$ is a weak solution of the incompressible Euler equations if
\begin{equation}\label{e:wie1}
\int_0^T\int_{\R^n}\partial_t\varphi\cdot v+\nabla\varphi:v\otimes v\,dxdt=0
\end{equation}
for all $\varphi\in C_c^{\infty}(\R^n\times (0,T);\R^n)$ with $\div\varphi=0$
and
\begin{equation}\label{e:wie2}
\int_0^T\int_{\R^n}v\cdot \nabla\psi\,dxdt=0 \qquad \mbox{for all $\psi\in C_c^{\infty}(\R^n\times (0,T))$}.
\end{equation}
When $v_0\in L^2(\R^n)$, the vectorfield $v$ is a weak solution of \eqref{e:Cauchy} if
\eqref{e:wie1} can be replaced by  
\begin{equation}\label{e:wie3}
\int_0^T\int_{\R^n} \partial_t\varphi\cdot v+\nabla\varphi:v\otimes v\,dxdt+\int_{\R^n} \varphi(x,0)\cdot v_0(x)\,dx=0
\end{equation}
for all $\varphi\in C_c^{\infty}(\R^n\times [0,T);\R^n)$ with $\div\varphi=0$.
\end{definition} 

%%%%%%%%%%%%%%%%%%%%%%%%%%%%%%%
\subsection{Weak solutions with compact support in time}
%%%%%%%%%%%%%%%%%%%%%%%%%%%%%%%
As already mentioned, the first nonuniqueness result for weak solutions of \eqref{e:Cauchy} is due to Scheffer in
his groundbreaking paper \cite{Scheffer93}. The main theorem of \cite{Scheffer93} states the existence of a nontrivial weak solution in $L^2(\R^2\times \R)$ with compact support in space and time. Later on Shnirelman in \cite{Shnirelman1} gave a different proof of the existence of a nontrivial weak solution in $L^2(\T^2\times\R)$ with compact support in time.  In these constructions
it is not clear if the solution belongs to the energy space (we refer to the next section for the relevant definition).
In the paper \cite{DS1} we provided a relatively simple proof of the following stronger statement.

\begin{theorem}[Non-uniqueness of weak solutions]\label{t:1}
There exist infinitely many compactly supported weak solutions of the incompressible Euler equations
in any space dimension. In particular there are infinitely many solutions $v\in L^\infty\cap L^2$ to \eqref{e:Cauchy}
for $v_0=0$ and arbitrary $n\geq 2$.  
\end{theorem}

%%%%%%%%%%%%%%%%%%%%%%%%%%%%%%%
\subsection{The Reynolds stress and subsolutions}\label{ss:reynolds}
%%%%%%%%%%%%%%%%%%%%%%%%%%%%%%%

The proof in \cite{DS1} is based on the notion of subsolution. In order to motivate its definition, let us recall the concept of Reynolds stress. 
It is generally accepted that the appearance of high-frequency oscillations in the velocity field is the main responsible for turbulent phenomena in
incompressible flows. One related major problem is therefore to understand the dynamics of the coarse-grained, in other words macroscopically averaged, velocity field.
If $\overline{v}$ denotes the macroscopically averaged velocity field, then it satisfies
\begin{equation}\label{e:Reynolds}
\begin{array}{l}   
\partial_t\overline{v}+\div(\overline{v}\otimes\overline{v}+ R)+\nabla\overline{p} =0\\
\div\overline{v} =0,
\end{array}
\end{equation}
where 
$$
R=\overline{v\otimes v}-\overline{v}\otimes\overline{v}\, .
$$
The latter quantity is called Reynolds stress and arises because the averaging does not commute
with the nonlinearity $v\otimes v$. On this formal level the precise definition 
of averaging plays no role, be it long-time averages, ensemble-averages or local space-time averages. The latter 
can be interpreted as taking weak limits and indeed,. Indeed weak limits of Leray solutions of the Navier-Stokes equations with vanishing viscosity 
have been proposed in the literature as a deterministic approach to turbulence (see \cite{Lax}, \cite{ChorinBook},
\cite{Bardos}, \cite{BardTiti}). 

A slightly more general version of this type of averaging follows the framework introduced by Tartar \cite{Tartar79, Tartar82}
and DiPerna \cite{Diperna85} in the context of conservation laws. We start by separating the linear 
equations from the nonlinear constitutive relations. Accordingly, we write \eqref{e:Reynolds} as
\begin{equation*}
\begin{array}{l}
\partial_t \overline{v}+\div \overline{u}+\nabla \overline{q} =0\\
\div \overline{v} =0,
\end{array}
\end{equation*}
where $\overline{u}$ is the traceless part of 
$\overline{v}\otimes \overline{v}+R$. 
Since one can write
$$
R=\overline{(v-\overline{v})\otimes(v-\overline{v})},
$$ 
it is clear that $R\geq 0$, i.e. $R$ is a symmetric positive semidefinite matrix. 
In terms of the coarse-grained variables $(\overline{v},\overline{u})$ this inequality can be written as
\begin{equation*}
\overline{v}\otimes\overline{v}-\overline{u}\leq \tfrac{2}{n}\overline{e}\,I,
\end{equation*}
where $I$ is the $n\times n$ identity matrix and
$$
\overline{e}=\overline{\tfrac{1}{2}|v|^2}
$$ 
is the macroscopic kinetic energy density. Motivated by these calculations, we define subsolutions as follows.
Since they will appear often, we introduce the notation $\S^{n\times n}_0$ for the vector space
of symmetric traceless $n\times n$ matrices.

\begin{definition}[Subsolutions]\label{d:subsolution}
Let $\overline{e}\in L^1_{loc}(\R^n\times (0,T))$ with $\overline{e}\geq 0$. A subsolution to the incompressible Euler equations with given kinetic energy density $\overline{e}$
is a triple 
$$
(v,u,q):\R^n\times (0,T)\to \R^n\times\S^{n\times n}_0\times\R
$$
with the following properties:
\begin{itemize}
\item $v\in L^2_{loc},\quad u\in L^1_{loc},\quad q$ is a distribution;
\item
\begin{equation}\label{e:LR}
\left\{\begin{array}{l}
\partial_t v+\div u+\nabla q =0\\
\div v =0,
\end{array}\right. \qquad \mbox{in the sense of distributions;}
\end{equation}
\item
\begin{equation}\label{e:CR}
v\otimes v-u\leq \tfrac{2}{n}\overline{e}\,I\quad\textrm{ a.e. .}
\end{equation}
\end{itemize}
\end{definition}

Observe that subsolutions automatically satisfy $\tfrac{1}{2}|v|^2\leq \overline{e}$ a.e. (the inequality follows
from taking the trace in \eqref{e:CR}). If in addition we have the
equality sign $\tfrac{1}{2}|v|^2=\overline{e}$ a.e., then the $v$ component of the subsolution is in fact a weak solution of the Euler equations. 
As mentioned above, in passing to weak limits (or when considering any other averaging process), the high-frequency oscillations in the velocity are responsible for the appearance of a non-trivial Reynolds stress. Equivalently stated, this phenomenon is responsible for the inequality sign in \eqref{e:CR}. 

\medskip

The key point in our approach to prove Theorem \ref{t:1} is that, starting from a subsolution, an appropriate
iteration process reintroduces the high-frequency oscillations. In the limit of this process one obtains weak solutions.
However, since the oscillations are reintroced in a very non-unique way, in fact this generates {\em several}
solutions from the same subsolution. A brief outline of the relevant iteration scheme will be given in Section \ref{s:applications}. 
In the next theorem we give a precise formulation of the previous discussion.

\begin{theorem}[Subsolution criterion]\label{t:criterion}
Let $\overline{e}\in C(\R^n\times (0,T))$ and $(\overline{v},\overline{u},\overline{q})$ a smooth, strict subsolution, i.e.
\begin{equation}\label{e:ass1}
(\overline{v},\overline{u},\overline{q})\in C^{\infty}(\R^n\times (0,T))\textrm{ satisfies \eqref{e:LR}}
\end{equation}
and
\begin{equation}\label{e:ass2}
\overline{v}\otimes \overline{v}-\overline{u}<\tfrac{2}{n}\overline{e}\qquad \mbox{on $\R^n\times (0,T)$.}
\end{equation}
Then there exist
infinitely many weak solutions $v\in L^{\infty}_{loc}(\R^n\times(0,T))$
of the Euler equations such that
\begin{eqnarray*}
\tfrac{1}{2}|v|^2&=&\overline{e},\\
p&=&\overline{q}-\tfrac{2}{n}\overline{e}
\end{eqnarray*}
for a.e.~$(x,t)$.
Infinitely many among these belong to $C ((0,T), L^2)$.
If in addition
\begin{equation}\label{e:initialdatum}
\overline{v}(\cdot,t)\rightharpoonup v_0(\cdot)\textrm{ in }L^2_{loc}(\R^n)\textrm{ as }t\to 0 ,
\end{equation}
then all the $v$'s so constructed solve \eqref{e:Cauchy}.
\end{theorem}

This theorem is indeed Proposition 2 of \cite{DS2} and from it we derive essentially all the results 
concerning the incompressible Euler equations 
which are surveyed in this paper. However, for a couple of cases we will need the following slightly stronger 
statement, which can be proved with the same methods (see also the Appendix of \cite{Szekelyhidi} for a more general statement):

\begin{proposition}\label{p:criterion_str}
The same conclusions of Theorem \ref{t:criterion} hold if \eqref{e:ass1} and \eqref{e:ass2} are
replaced by the following weaker assumptions
\begin{enumerate}
\item $(\overline{v},\overline{u},\overline{q})$ is a continuous subsolution;
\item for all $(x,t)\in\R^n\times(0,T]$
$$
\frac{1}{2}|\overline{v}|^2<\overline{e}\quad\Rightarrow\quad\overline{v}\otimes \overline{v}-\overline{u}<\tfrac{2}{n}\overline{e}
$$ 
\item The domain $\{(x,t):\,\frac{1}{2}|\overline{v}|^2<\overline{e}\}\subset \R^n\times (0,T)$
has nonempty interior and the boundary of each time-slice has $0$ Lebesgue measure.
\end{enumerate}
\end{proposition}

Condition (2) amounts to the requirement that in the open subset of $ \R^n\times (0,T)$ where $(\overline{v},\overline{u},\overline{q})$ is not a solution, it should actually be a strict subsolution. 

\begin{proof}[Sketch proof of Proposition \ref{p:criterion_str}]
The argument is exactly the same as that of Proposition 2 of \cite{DS2}. A close inspection of that argument
shows that it uses only the continuity of the subsolution. It remains therefore to show that
the open set $\Omega\times (0,T)$ in Proposition 2 of \cite{DS2} can be substituted with the more general 
$U:=\{\bar{v}\otimes\bar{v} - \bar{u} < \frac{2}{n}\bar{e}\}$. In the proof of \cite{DS2} the assumption 
$U=\Omega\times (0,T)$ plays a role 
only in the construction of the grid of Subsection 4.5 of \cite{DS2}. In order to handle the case considered here, it
suffices to use the same cubical partition of $\R^n \times \R$ and to perform the perturbations only on those
closed cubes which are contained in $U$.
\end{proof}

%%%%%%%%%%%%%%%%%%%%%%%%%%%%%%%
\subsection{Global existence of weak solutions}
%%%%%%%%%%%%%%%%%%%%%%%%%%%%%%%

One way to utilize Theorem \ref{t:criterion} is to start with an initial datum $v_0\in L^2_{loc}(\R^n)$, and
construct a smooth triple $(\overline{v},\overline{u},\overline{q})\in C^{\infty}(\R^n\times (0,T))$ 
solving \eqref{e:LR} with initial datum $v_0$ in the sense of \eqref{e:initialdatum}. We further define 
$$
e(v,u):=\tfrac{n}{2}\lambda_{max}(v\otimes v-u),
$$
where $\lambda_{\max}$ denotes the largest eigenvalue. It is then obvious that $(\overline{v},\overline{u},\overline{q})$  is a smooth, strict subsolution for 
$$
\overline{e}(x,t):=e\bigl(\overline{v}(x,t),\overline{u}(x,t)\bigr)+\min (t,\tfrac{1}{t})
$$
Of course there are many ways of constructing such a subsolution, since the system \eqref{e:LR} is underdetermined (there is no 
evolution equation for $\overline{u}$!). This observation is closely related to the well known closure problem in turbulence.

By constructing a subsolution with bounded energy, E.~Wiedemann in \cite{Wiedemann} recently obtained the following:
\begin{corollary}[Global existence for weak solutions]
Let $v_0\in L^2(\T^n)$ be a solenoidal vectorfield. Then there exist infinitely many 
global weak solutions \eqref{e:Cauchy} with bounded energy, i.e. such that
$$
E(t)=\frac{1}{2} \int_{\T^n}|v(x,t)|^2\,dx
$$
is bounded. Moreover $E(t)\to 0$ as $t\to \infty$. 
\end{corollary}
It should be noted, however, that for these solutions the energy $E(t)$ does not 
converge to $\frac{1}{2}\|v_0\|_{L^2}^2$ as $t\to 0$: the energy will, in general, have
an instantaneous jump. We will return to this issue in the next section.

%%%%%%%%%%%%%%%%%%%%%%%%%%%%%%%
\subsection{The coarse-grained flow and measure-valued solutions}
%%%%%%%%%%%%%%%%%%%%%%%%%%%%%%%

Following the idea that a subsolution $(\overline{v},\overline{u},\overline{q})$ represents the averaged (or coarse-grained)
velocity, stress tensor and pressure respectively, it is natural to expect that, for any given subsolution, there is a sequence of
weak solutions $v_k$ with $v_k\rightharpoonup \overline{v}$ in $L^2_{loc}$ as $k\to\infty$. Indeed, this is a consequence of the proof 
of Theorem \ref{t:criterion}: one obtains a set of weak solutions which is dense in the space of subsolutions in the weak $L^2$ topology. For details see Section \ref{s:applications}.
This can be made even more precise with the concept of Young measures. 

Let us recall that, given a sequence $v_k\in L^2_{loc} (\R^n\times [0,T))$, there exists a subsequence (not relabeled) and a parametrized probability measure $\nu_{x,t}$ on $\R^n\times [0,T)$, called the associated Young measure, such that
\begin{equation}\label{e:YM}
f(v_k)\overset{*}{\rightharpoonup} \int_{\R^n}f\,d\nu_{x,t}\quad\textrm{ in }L^{\infty}_{loc}
\qquad \mbox{for all bounded continuous $f$.}
\end{equation}
Thus, Young measures record the one-point statistics of oscillations
in weakly convergent sequences. 

In order to capture both high-frequency oscillations as well as possible concentrations for Euler flows, DiPerna and Majda 
introduced the concept of measure-valued solutions.
%for the incompressible Euler equations. 
With this generalization, one can make
sense of the limit \eqref{e:YM} also for test functions $f$ which are not necessarily bounded, in particular for the stress tensor $f_{ij}(v)=v^iv^j$. 
Here we follow Alibert and Bouchitt\'e \cite{AlibertBouchitte}. For such test-functions the limit \eqref{e:YM} takes the form
\begin{equation}\label{e:DM}
f(v_k)\,dxdt\overset{*}{\rightharpoonup} \biggl(\int_{\R^n}f\,d\nu_{x,t}\biggr)\,dxdt+\biggl(\int_{S^{n-1}}f^{\infty}\,d\nu^{\infty}_{x,t}\biggr)\,\lambda(dxdt)
\end{equation}
where the convergence is in the sense of measures and $f^{\infty}(\xi):=\lim_{s\to\infty}\tfrac{f(s\xi)}{s^2}$ is the $L^2$ recession function of $f$. The triple $(\nu,\lambda, \nu^\infty)$ is made of
\begin{itemize}
\item the oscillation measure $\nu_{x,t}$ which  
is a parametrized probability measure on $\R^n$;
\item the concentration measure $\lambda$, which is a nonnegative Radon measure on $\R^n\times(0,T)$;
\item the concertration-angle measure $\nu^{\infty}_{x,t}$ which is a parametrized probability measure on $S^{n-1}$.
\end{itemize}
Note that for bounded $f$ the formula in \eqref{e:DM} reduces to that in \eqref{e:YM}, because $f^\infty = 0$.

\begin{definition}[Measure-valued solutions]
A measure-valued solution of the incompressible Euler equations is a triple $(\nu_{x,t},\lambda,\nu_{x,t}^{\infty})$ such that the following two identities hold for all 
$\varphi\in C_c^{\infty}((0,T)\times\R^n;\R^n)$ with $\div\varphi=0$ and for all
$\psi\in C_c^{\infty}((0,T)\times\R^n)$:
\begin{equation}\label{e:mvie1}
\begin{split}
\int_0^T\int_{\R^n}\partial_t\varphi\cdot \langle \xi,\nu_{x,t}\rangle&+\nabla\varphi:\langle \xi\otimes \xi,\nu_{x,t}\rangle\,dxdt\\
&+\int_0^T\int_{\R^n}\nabla\varphi:\langle \theta\otimes\theta,\nu^{\infty}_{x,t}\rangle\,d\lambda=0\, ,
\end{split}
\end{equation}
\begin{equation}\label{e:mvie2}
\int_0^T\int_{\R^n}\langle \xi,\nu_{x,t}\rangle\cdot \nabla\psi\,dxdt=0.
\end{equation}
For $v_0\in L^2(\R^n)$, the triple $(\nu_{x,t},\lambda,\nu_{x,t}^{\infty})$ is a measure-valued solution of 
\eqref{e:Cauchy} if  \eqref{e:mvie1} can be replaced by  
\begin{equation}\label{e:mvie3}
\begin{split}
\int_0^T\int_{\R^n}&\partial_t\varphi\cdot \langle \xi,\nu_{x,t}\rangle+\nabla\varphi:\langle \xi\otimes \xi,\nu_{x,t}\rangle\,dxdt\\
&+\int_0^T\int_{\R^n}\nabla\varphi:\langle \theta\otimes\theta,\nu^{\infty}_{x,t}\rangle\,d\lambda+\int_{\R^n} \varphi(x,0)\cdot v_0(x)\,dx=0
\end{split}
\end{equation}
for all $\varphi\in C_c^{\infty}([0,T)\times\R^n;\R^n)$ with $\div\varphi=0$.
\end{definition}
In these formulas
\begin{equation}\label{e:mvv}
\overline{v}:=\langle \xi,\nu_{x,t}\rangle=\int_{\R^n}\xi\,d\nu_{x,t}(\xi)
\end{equation}
is the average (coarse-grained) velocity field, and similarly 
\begin{equation}\label{e:mvu}
\begin{split}
\overline{u}+\tfrac{2}{n}\overline{e}I&:=\langle \xi\otimes\xi,\nu_{x,t}\rangle+\langle\theta\otimes\theta,\nu_{x,t}^{\infty}\rangle\,\lambda\\
&=\int_{\R^n}\xi\otimes\xi\,\nu_{x,t}(\xi)+\int_{S^{n-1}}\theta\otimes\theta\,\nu_{x,t}^{\infty}(\theta)\,\lambda(dxdt).
\end{split}
\end{equation}
is the average stress tensor (recall from subsection \ref{ss:reynolds} that by $\overline{u}$ we denote the traceless part of the stress tensor). 
Note that in general the latter is a measure rather than a locally integrable function, because of possible $L^2$ concentrations.
In \cite{DM2} DiPerna and Majda showed that any sequence of Leray solutions of the Navier-Stokes equations converges in the vanishing viscosity limit to a measure-valued solution. 

Measure-valued solutions 
give rise to subsolutions 
as in Definition \ref{d:subsolution}. Conversely, given a subsolution $(\overline{v},\overline{u},\overline{q})$, it is trivial to extend it to
a measure-valued solution by finding for a.e.~$(x,t)$ probability measures $\nu_{x,t}$ such that \eqref{e:mvv} and \eqref{e:mvu} are satisfied.
In this regard it is important to note that in the definition of measure-valued solutions there are no {\it microscopic} constraints, that is, contraints on the
distributions of the probability measures. This is very different from other contexts where Young measures have been used, such as conservation laws in one space dimension \cite{Tartar82, Diperna85}: in these and similar references the Young measures satisfy
additional microscopic constraints in the form of commutativity relations (for instance as a consequence of the div-curl lemma
applied to the generating sequence).

Measure-valued solutions are a very weak notion, with a huge scope for unnatural non-uniqueness. However,
the stronger notion of weak solution in Definition \ref{d:weaksolution} 
actually exhibits this very same non-uniqueness, as witnessed by the following theorem, proved in \cite{SzWie}.

\begin{theorem}[Weak solutions as 1-point statistics]\label{t:mv}
Given a measure-valued solution $(\nu_{x,t},\lambda,\nu_{x,t}^{\infty})$ of the incompressible Euler equations, there exists a sequence of weak solutions $v_k$ with bounded energy such that \eqref{e:DM} holds.
\end{theorem}

%%%%%%%%%%%%%%%%%%%%%%%%%%%%%%%
\section{Energy}
%%%%%%%%%%%%%%%%%%%%%%%%%%%%%%%
In the previous section we have seen that weak solutions of the Euler 
equations are
in general highly non-unique, at least if we interpret weak solutions in the usual distributional sense of Definition \ref{d:weaksolution}. In particular the kinetic energy density $\tfrac{1}{2}|v|^2$ can be prescribed as an independent quantity. It is therefore quite remarkable that, despite this high flexibility, the additional requirement
that the energy $E(t)=\tfrac{1}{2}\int_{\R^n}|v|^2\,dx$ be non-increasing already suffices to single out the unique classical solution when it exists. 

\begin{theorem}[Weak-strong uniqueness]\label{t:ws}
Let $v\in L^\infty ([0,T), L^2 (\R^n))$ be a weak solution of \eqref{e:Cauchy} with the additional property
that $\nabla v + \nabla v^T\in L^\infty$. Assume that $(\nu, \lambda, \nu^\infty)$ is a measure-valued 
solution of \eqref{e:Cauchy} satisfying 
\begin{equation}\label{e:energy1}
\int_{\R^n} \int_{\R^n} |\xi|^2 d\nu_{x,t} (\xi)\, dx + \int_{\R^n} d\lambda_t (x) \;\leq\; \int_{\R^n} |v_0|^2 (x)\, dx
\qquad \mbox{for a.e. $t$.}
\end{equation}
Then $(\nu, \lambda, \nu^\infty)$ coincides with $v$ as long as the latter exists, i.e.
$$
\nu_{x,t} = \delta_{v (x,t)} \;\;\mbox{for a.a. $(x,t)\in \R^n\times (0,T)$
 and }\;\; \lambda\equiv 0 \;\;\mbox{on $\R^n\times (0,T)$.}
$$
\end{theorem}

This theorem recently appeared in \cite{BDS}, building upon 
ideas of \cite{Brenier:2000,BrenierGrenier}, where the authors dealt with the energy of 
measure-valued solutions to the Vlasov-Poisson system. More precisely, the proof of 
\cite{BDS} yields the following information:
if $\nu_{x,t}$ satisfies \eqref{e:energy1}, then
$$
\bar{v}(x,t) := \int_{\R^n} \xi\, d\nu_{x,t} (\xi) \quad \left( = \langle \xi, \nu_{x,t}\rangle \right)
$$
is a {\em dissipative solution} of \eqref{e:Cauchy} in the sense of Lions (see \cite{Lions}).
In fact, Lions introduced the latter notion to gain back the weak-strong uniqueness
while retaining the weak compactness
properties of the DiPerna-Majda solutions. Theorem \ref{t:ws} shows that this can
be achieved in the framework of
DiPerna and Majda by simply adding the natural energy constraint \eqref{e:energy1}.

\subsection{Admissible weak solutions} It is easy to see that $C^1$ solutions of
the incompressible Euler equations satisfy the following identity, which expresses
the conservation of the kinetic energy in a local form
\begin{equation}\label{e:energy2}
\partial_t \frac{|v|^2}{2} + {\rm div}\, \left(\left(\frac{|v|^2}{2} + p\right) v\right
) = 0\, .
\end{equation}
Integrating \eqref{e:energy2} in space we formally get the conservation of the total kinetic
energy
\begin{equation}\label{e:energy3}
\frac{d}{dt} \int_{\R^n} \frac{|v|^2}{2} (x,t)\, dx = 0\, .
\end{equation}
These identities suggest that the notion of weak solution to \eqref{e:Cauchy}
can be  complemented with several admissibility criteria, which we list here:
\begin{itemize}
\item[(a)]
\begin{equation*}%\label{e:energy_a}
\int |v|^2 (x,t)\, dx \leq \int |v_0|^2 (x)\, dx \qquad \mbox{for a.e. $t$}.
\end{equation*}
\item[(b)]
\begin{equation*}%\label{e:energy_b}
\int |v|^2 (x,t)\, dx \leq \int |v|^2 (x,s)|\, dx \qquad \mbox{for a.e. $t>s$}.
\end{equation*}
\item[(c)] If in addition $v\in L^3_{loc}$, then 
 \begin{equation*}%\label{e:energy_c}
\partial_t \frac{|v|^2}{2} + {\rm div}\, \left(\left(\frac{|v|^2}{2} + p\right) v\right
) \leq 0\, 
\end{equation*}
in the sense of distributions (note that, since $-\Delta p = {\rm div}\, {\rm div}\, v\otimes v$,
the product $|v|^2p$ is well-defined by the Calderon-Zygmund inequality).
\end{itemize}
Condition (c) has been proposed for the first time by Duchon and Robert in
\cite{DuchonRobert} and it resembles the admissibility criteria which are 
popular in the literature on hyperbolic conservation laws.

Next, denote by $L^2_w (\R^n)$ the space $L^2(\R^n)$
endowed with the weak topology. We recall
that any weak solution of \eqref{e:Cauchy} can be modified on a set of measure
zero so to get $v\in C ([0,T), L^2_w (\R^n)$ (this is a common feature of evolution
equations in conservation form; see for instance Theorem 4.1.1 of \cite{DafermosBook}). 
Consequently $v$ has a well-defined trace at every
time and the requirements (a) and (b) can therefore be strengthened 
in the following sense:
\begin{itemize}
\item[(a')]
\begin{equation*}%\label{e:energy_a'}
\int |v|^2 (x,t)\, dx \leq \int |v_0|^2 (x)|\, dx \qquad \mbox{for {\em every} $t$}.
\end{equation*}
\item[(b')]
\begin{equation*}%\label{e:energy_b'}
\int |v|^2 (x,t)\, dx \leq \int |v|^2 (x,s)|\, dx \qquad \mbox{for {\em every} $t>s$}.
\end{equation*}
\end{itemize}

However, none of these criteria restore the uniqueness of weak solutions.

\begin{theorem}[Non-uniqueness of admissible weak solutions]\label{t:energy_sucks}
There exist initial data $v_0\in L^\infty\cap L^2$ for which there are 
infinitely many bounded solutions of \eqref{e:Cauchy} which are strongly $L^2$-continuous (i.e.
$v\in C ([0,\infty), L^2 (\R^n))$) and satisfy (a'), (b') and (c). 

The conditions (a'), (b') and (c) hold with the equality sign 
for infinitely many of these solutions, whereas for infinitely many other
they hold as {\em strict} inequalities.  
\end{theorem}

This theorem is from \cite{DS2}. The second statement generalizes the intricate
construction of Shnirelman in \cite{Shnirelmandecrease}, which produced the first example  
of a weak solution in 3D of \eqref{e:Cauchy} with strict inequalities
in (a) and (b). 

%%%%%%%%%%%%%%%%%%%%%%%%%%%%%%%
\subsection{Wild initial data}\label{s:wild}
%%%%%%%%%%%%%%%%%%%%%%%%%%%%%%%
The initial data $v_0$ as in Theorem \ref{t:energy_sucks} are obviously not
regular, since for regular initial data the local existence theorems and the weak-strong uniqueness (Theorem \ref{t:ws}) ensure
local uniqueness under the very mild condition (a). One might therefore ask how large is the
set of these ``wild'' initial data. A consequence of our methods is the following
density theorem (cp. Theorem 2 in \cite{SzWie}).

\begin{theorem}[Density of wild initial data]
The set of initial data $v_0$ for which the conclusions of Theorem \ref{t:energy_sucks} holds
is dense in the space of $L^2$ solenoidal vectorfields.
\end{theorem}

Another surprising corollary is that the usual shear flow is a "wild initial data". More precisely,
consider the following solenoidal vector field in $\R^2$
\begin{equation}\label{e:shear}
v_0 (x) = 
\left\{
\begin{array}{ll}
(1,0) & \mbox{if $x_2>0$}\\
(-1, 0) & \mbox{if $x_2<0$}.
\end{array}\right.
\end{equation}
or the following solenoidal vector field in $\T^2 ={\mathbb S}^1\times {\mathbb S}^1$:
\begin{equation}\label{e:shear2}
v_0 (x) = 
\left\{
\begin{array}{ll}
(1,0) & \mbox{if $\theta_2\in (-\pi, 0)$}\\
(-1, 0) & \mbox{if $\theta_2\in (0, \pi)$}.
\end{array}\right.
\end{equation}

\begin{theorem}[The vortex-sheet is wild]\label{t:shear} 
For $v_0$ as in \eqref{e:shear}.
there are infinitely many weak solutions of \eqref{e:Cauchy} on $\R^2\times [0, \infty)$ 
which satisfy (c). For $v_0$ as in \eqref{e:shear2} there are
infinitely many weak solutions of \eqref{e:Cauchy} on $\T^2\times [0,\infty)$ which 
satisfy (c), (a') and (b').
\end{theorem}

Theorem \ref{t:shear} is proved in \cite{SzVortex} using Proposition \ref{p:criterion_str} and hence
the proof essentially amounts to showing the existence of a suitable subsolution. The construction of such subsolution
follows an idea introduced in \cite{Szekelyhidi} for the incompressible porous media equation, see Theorem \ref{t:lsz-ipm}
below. 

Since the various additional requirements discussed above do not ensure uniqueness of the solution even for
this very natural initial condition, Theorem \ref{t:shear} raises the following natural question: {\it is there a way to single out 
a unique, physically relevant solution?} In two space dimensions one could further impose that the vorticity is 
a measure, leading to the well-known problem of uniqueness for the vortex sheet (we
note in passing that our methods do not seem to apply to the vorticity formulation of $2$-d Euler:
compare with the discussion in Section \ref{ss:active}). Two other popular criteria
considered in the literature for hyperbolic conservation laws are
\begin{enumerate}
\item the vanishing viscosity limit,
\item the maximally dissipative solution.
\end{enumerate}
For scalar conservation laws they both single out the unique entropy solution, see \cite{DafermosRate,BlaserRiviere}. For the 2-dimensional incompressible Euler system the situation
is surely more complicated.
As it happens with Theorem  \ref{t:energy_sucks}, some of the solutions constructed in \cite{SzVortex}
preserve the energy, whereas some other are dissipative. On the other hand, it is easy to see that the vanishing viscosity limit for initial data as in \eqref{e:shear} or \eqref{e:shear2} is the stationary solution, which is obviously conservative. Therefore, even if the two criteria singled out unique weak solutions, they would be two {\it different} ones.

In \cite{SzVortex} it is also shown that, for the 2-dimensional torus $\T^2$, the maximal dissipation rate reachable with the proof of Theorem \ref{t:shear} is
$$
\max\tfrac{dE}{dt}=-1/6.
$$ 
It is, however, not clear whether there is a solution with this precise dissipation rate and, if it exists, whether it is unique. 

%%%%%%%%%%%%%%%%%%%%%%%%%%%%%%%
%%%%%%%%%%%%%%%%%%%%%%%%%%%%%%%
\section{Applications to other equations}\label{s:applications}

The ideas introduced in the previous sections apply to many other nondissipative
systems of evolutionary partial differential equations. We start with some general considerations, and refer the reader for more details on
this general framework to \cite{KMS02,MuLecturenotes}.

Following Tartar \cite{Tartar82}, we consider general systems in a domain $\D\subset\R^d$ of the form
 \begin{eqnarray}
 \sum_{i=1}^dA_i\partial_iz=0&&\textrm{ in }\D\label{e:LRbis}\\
 z(y)\in K&&\textrm{ a.e. }y\in\D\label{e:CRbis}
\end{eqnarray}
where
\begin{itemize}
\item $z:\D\subset\R^d\to\R^N$ is the unknown state variable,
\item $A_i$ are constant $m\times N$ matrices
\item and $K\subset\R^N$ is a closed set.
\end{itemize} 
{\it Plane waves} are solutions of \eqref{e:LRbis} of the form
\begin{equation}\label{planewave}
z(x)=ah(x\cdot\xi),
\end{equation}
where $h:\R\to\R$. The 
{\it wave cone} $\Lambda$ is given by the states $a\in\R^N$ such that for any choice of the profile $h$
the function \eqref{planewave} solves \eqref{e:LRbis}, that is,
\begin{equation}\label{e:wavecone}
\Lambda:=\left\{a\in\R^N:\,\exists\xi\in\R^d\setminus\{0\}\quad
\mbox{with}\quad \sum_{i=1}^d\xi_i A_ia=0\right\}.
\end{equation}
The oscillatory behavior of solutions to the nonlinear problem is 
then determined by the compatibility of the set $K$ with the 
cone $\Lambda$. This compatibility is expressed in terms of a suitable
concept of $\Lambda$-convex hull $K^{\Lambda}$ (for the precise definition,
see Section \ref{s:hulls}). Modulo technical details,
the subsolutions from
Definition \ref{d:subsolution} are solutions $z$ of the linear relations 
\eqref{e:LRbis} which satisfy the relaxed condition $z\in K^{\Lambda}$. 

The idea of convex integration is to reintroduce oscillations by adding
suitable localized versions of \eqref{planewave} to the subsolutions and to
recover a solution of \eqref{e:LRbis} - \eqref{e:CRbis} iterating this process. The upshot is that in a Baire generic
sense, most solutions of the ``relaxed system'' are actually solutions of the
original system.

There are various different forms of implementing convex integration in this general framework, see for instance \cite{MS99,DacorognaMarcellini97,Sychev,Bernd}.
Common to all approaches is that one is working in a space of subsolutions in which highly oscillatory perturbations
are possible. An elegant way of formalizing this was introduced by B.~Kirchheim in \cite{Bernd} Section 3.3.
We recall the main steps. 

The space of subsolutions arises from a nontrivial open set $\mathcal{U}\subset \R^N$
satisfying the following pertubation property (cp. for instance with Proposition 2.2 in \cite{DS1} and the proof
of Theorem 4.1 in the same paper).

\bigskip

\noindent{\bf Perturbation Property (P): }
There is a continuous function $\eps: \R^+\to \R^+$ with $\eps(0)=0$ with the following property
For every $z\in \mathcal{U}$
there is a sequence of solutions $z_j\in C^\infty_c (B_1)$ of \eqref{e:LRbis} such that
\begin{itemize}
\item $z_j\overset{*}{\rightharpoonup} 0$ in $L^{\infty}(\R^d)$;
\item $z+z_j (y)\in \mathcal{U}\quad \forall y\in\R^d$;
\item $\int |z_j (y)|^2 dy \geq \eps ({\rm dist}\, (z,K))$. 
\end{itemize} 

\bigskip

Next, let 
$$
X_0=\bigl\{z\in C^\infty_c (\mathcal{D}):\,\textrm{\eqref{e:LRbis} holds and }z (y)\in \mathcal{U}\textrm{ for all }y\in \mathcal{D}\bigr\}
$$
and let $X$ be the closure of $X_0$ in $L^{\infty}(\mathcal{D})$ with respect to the weak$^*$ topology. 
Assuming that $K$ is bounded, the set $X$ is bounded
in $L^\infty$ and the weak$^*$ topology is therefore metrizable on $X$. 

An easy covering argument, together with property (P), results in the following lemma:
\begin{lemma}
There is a continuous function $\tilde\eps:\R^+\to\R^+$ with $\tilde\eps(0)=0$ such that, for every $z\in X_0$
there is a sequence $z_j\in X_0$ with
$$
\int_{\mathcal{D}}|z_j-z|^2\,dy\geq \tilde\eps\left(\int_{\mathcal{D}}\dist(z(y),K)\,dy\right)
$$
\end{lemma}
Since the map $z\mapsto\int_{\mathcal{D}}|z|^2\,dy$ is a Baire-1 function on $X$, an easy 
application of the Baire category theorem gives that the subset of $z\in X$ satisfying 
\eqref{e:CRbis} is Baire-generic in $X$. 

The argument just sketched yields weak solutions to \eqref{e:LRbis}-\eqref{e:CRbis} (assuming that (P) holds for some $\mathcal{U}$) which are zero on the boundary $\partial\mathcal{D}$ in the trace sense with respect to the operator in \eqref{e:LRbis}. In particular these weak solutions are extendable by zero to $\R^d\supset\mathcal{D}$. In applications to evolution equations $\mathcal{D}$ is a space-time domain, say $\mathcal{D}=\R^n\times (0,T)$, and thus this argument yields weak solutions with compact time support, as in \cite{Scheffer93,CFG,Shvydkoy}. For the construction of weak solutions with arbitrary initial data, in particular for the construction of admissible weak solutions, refinements of this argument are necessary. A more detailed exposition for such cases is presented in the Appendix of \cite{Szekelyhidi}.

In the following we survey some examples.

%%%%%%%%%%%%%%%%%%%%%%%%%%%%%%%
\subsection{Compressible Euler} 
%%%%%%%%%%%%%%%%%%%%%%%%%%%%%%%

As a byproduct of our analysis of
the incompressible Euler system,
non--uniqueness theorems for admissible solutions 
of the so-called $p$--system were proved in \cite{DS2}.
The system of isentropic gas dynamics in Eulerian 
coordinates is the oldest
hyperbolic system of conservation laws and consists of $n+1$ equations
in $n$ space dimensions. The unknowns are 
the density $\rho$ and the velocity $v$ of the gas. The equation are
\begin{equation}\label{e:psistema}
\left\{\begin{array}{l}
\partial_t \rho + {\rm div}_x (\rho v) \;=\; 0\\
\partial_t (\rho v) + {\rm div}_x (\rho v\otimes v) + \nabla [ p(\rho)]\;=\; 0\\
\rho (0, \cdot)\;=\; \rho^0\\
v (0, \cdot)\;=\; v^0\, 
\end{array}\right.
\end{equation}
(cf.~(3.3.17) in \cite{DafermosBook} and Section 1.1 of \cite{SerreBook} p7).
The pressure $p$ is a function of $\rho$, which is determined from
the constitutive thermodynamic relations of the gas in question
and satisfies the assumption $p' >0$. A typical example is $p (\rho )= k \rho^\gamma$,
with constants $k>0$ and $\gamma>1$,
which gives the constitutive relation for a polytropic gas
(cf.~(3.3.19) and (3.3.20) of \cite{DafermosBook}).

Weak solutions of \eqref{e:psistema} are bounded functions in $\R^n$, which solve
it in the sense of distributions. Admissible solutions have to satisfy an additional
inequality, coming from the conservation law for the energy of the system.

\begin{definition}\label{d:admissible} A weak solution of
\eqref{e:psistema} is a pair $(\rho, v)\in L^{\infty}(\R^n)$ such that the following identities
hold for every test function $\psi,\varphi\in C^\infty_c (\R^n\times [0,\infty[)$:
\begin{equation}\label{e:test1}
\int_0^{\infty}\int_{\R^n} 
\Bigl[\rho \partial_t \psi + \rho v \cdot \nabla_x \psi\Bigr]\,dx\,dt\;+\;
\int_{\R^n} \rho^0 (x)\, \psi (x,0)\, dx=0,
\end{equation}
\begin{equation}\label{e:test2}
\int_0^{\infty}\int_{\R^n}
\Bigl[\rho v\cdot \partial_t\varphi + \rho \langle v\otimes v, \nabla \varphi\rangle\Bigr]\,dx\,dt
\;+\; \int_{\R^n} \rho^0 (x) v^0 (x)\cdot \varphi (x,0)\, dx\, =0.
\end{equation}

Consider the energy $\eps: \R^{+}\to \R$ given through
the law $p (r)= r^2 \eps' (r)$. A weak solutions of \eqref{e:psistema} is
admissible if the following inequality
holds for every nonnegative $\psi\in C^\infty_c (\R^n\times\R)$:
\begin{eqnarray}
&&\int_0^{\infty}\int_{\R^n} \left[\left(\rho \eps (\rho) + 
\frac{\rho |v|^2}{2}\right)
\partial_t\psi + \left(\rho \eps (\rho) + \frac{\rho |v|^2}{2} +
p (\rho)\right)v\cdot \nabla_x\psi\right]\nonumber\\
&+& \int_{\R^n} \left(\rho^0 \eps (\rho^0) + 
\frac{\rho^0 |v^0|^2}{2}\right)
\psi (\cdot,0)\;\geq\; 0\, .\label{e:entropy2}
\end{eqnarray}
\end{definition}

The following nonuniqueness result was proved in \cite{DS2}.

\begin{theorem}[Non-uniqueness for the p-system]\label{t:psistema}
Let $n\geq 2$. Then, for any given function $p$, 
there exist bounded initial data $(\rho^0, v^0)$ with
$\rho^0\geq c>0$ for which there are infinitely many bounded admissible
solutions $(\rho, v)$ of \eqref{e:psistema} with $\rho\geq c>0$.
\end{theorem}

A variant of this theorem has been recently shown by E.~Chiodaroli using
the same techniques (see \cite{Chiodaroli}). Chiodaroli's Theorem highlights that the main role in this loss
of uniqueness is played by the velocity field.

\begin{theorem}[Non-uniqueness with arbitrary density]\label{t:chiodaroli}
For every periodic $\rho^0\in C^1$ with $\rho^0\geq c>0$ there exists an initial velocity $v^0\in L^\infty$
and a time $T>0$ such that there are infinitely many bounded admissible solutions
$(\rho, v)$ of \eqref{e:psistema} on $\R^n\times [0,T[$, all with density bounded away from $0$.
\end{theorem}

%%%%%%%%%%%%%%%%%%%%%%%%%%%%%%%
\subsection{Active scalar equations} \label{ss:active}
%%%%%%%%%%%%%%%%%%%%%%%%%%%%%%%

Active scalar equations are a class
of systems of evolutionary partial differential equations in $n$ space dimensions. The unknowns
are the ``active'' scalar function $\theta$ and the velocity $v$, which, for simplicity, is a divergence-free
vector field. The equations are
\begin{equation}\label{e:ac_pde}
\left\{\begin{array}{l}
\partial_t \theta + v\cdot \nabla_x \theta \;=\; 0 \\
{\rm div}_x v \;=\; 0
\end{array}\right.
\end{equation}
and $v$ and $\theta$ are coupled by an integral operator, namely
\begin{equation}\label{e:ac_int}
v \;=\; T [\theta]\, .
\end{equation}
Several systems of partial differential equations in fluid dynamics fall into this class.

We rewrite \eqref{e:ac_pde} and \eqref{e:ac_int}, in the spirit of Section \ref{ss:reynolds}
(see also \eqref{e:LR} and \eqref{e:CR}), as the system of linear relations 
\begin{equation}\label{e:ac_lr}
\left\{\begin{array}{ll}
\partial_t \theta + {\rm div}_x q \;=\; 0\\
{\rm div}_x v \;=\; 0\\
v = T [\theta]
\end{array}\right.
\end{equation}
coupled with the nonlinear constraint
\begin{equation}\label{e:ac_nc}
q = \theta v\, .
\end{equation}
The initial value problem for the system \eqref{e:ac_lr}-\eqref{e:ac_nc} amounts to
prescribing $\theta (x,0)=\theta_0 (x)$. 

As described at the beginning of this section, a key point is that the linear relations
\eqref{e:ac_lr} admit a large set of plane wave solutions. Note that these linear relations
are not strictly speaking of the form \eqref{e:LRbis} and in order to define a suitable analog
of the plane waves in this setting we assume that the linear operator $T$ is translation invariant.
Let $m(\xi)$ be its corresponding Fourier multiplier. Then we require in addition that
\begin{equation}\label{e:0hom}
m (\xi)\textrm{ is $0$-homogeneous}
\end{equation}
so that \eqref{e:ac_lr} has the same scaling invariance as \eqref{e:LRbis}.
Furthermore the
constraint ${\rm div}_x v = 0$ implies that 
\begin{equation}\label{e:divfree}
\xi\cdot m (\xi) = 0.
\end{equation}  

An important remark at this point is that the $0$-homogeneity of $m$ excludes the vorticity
formulation of the $2$-dimensional incompressible Euler equations: indeed convex integration
does not seem to apply to this situation because the highest order derivatives, namely the
vorticity itself, appear linearly in the equation. In a geometric context this issue has been raised 
by Gromov in Section 2.4.12 of \cite{Gromov} and to a certain extent analysed in \cite{Spring}.
Instead, the $0$-homogenity ensures that $\theta$ and $v$ are of the same order.

In spite of this restriction, several interesting equations fall into this category. Perhaps the
best known examples are the surface quasi geostrophic and the incompressible
porous medium equations, corresponding, respectively to  
\begin{eqnarray}
m (\xi) &=& i |\xi|^{-1} (-\xi_2, \xi_1) \label{e:SQG} 
\quad\mbox{and}\\
m (\xi) &=& |\xi|^{-2} ( \xi_1\xi_2, - \xi_1^2)\label{e:IPM}\, .
\end{eqnarray}
In \cite{CFG} Cordoba, Faraco and Gancedo proved 
\begin{theorem}\label{t:CFG}
Assume $m$ is given by \eqref{e:IPM}. Then there exist infinitely many weak solutions of 
\eqref{e:ac_lr} and \eqref{e:ac_nc} in $L^\infty (\T^2\times [0, +\infty[)$ with $\theta_0=0$.
\end{theorem}
This was generalized by Shvydkoy in \cite{Shvydkoy} to all even $m(\xi)$ satisfying a mild
additional regularity assumption, namely
\begin{theorem}\label{t:Shvid}
Assume $m$ is even, $0$ homogenous and 
the set
\[
\{m(\xi): m|_{\mathbb{S}^{n-1}} \mbox{ 
is a regular immersion around $\xi$}\}
\]
 spans $\R^n$. 
Then there exist infinitely many weak solutions of 
\eqref{e:ac_lr} and \eqref{e:ac_nc} in $L^\infty (\T^2\times [0, +\infty[)$ with $\rho_0=0$.
\end{theorem}

\bigskip

%%%%%%%%%%%%%%%%%%%%%%%%%%%%%%%
\subsection{Laminates and hulls}\label{s:hulls} 
%%%%%%%%%%%%%%%%%%%%%%%%%%%%%%%

The proofs of Theorems \ref{t:CFG} and \ref{t:Shvid} in \cite{CFG,Shvydkoy} use some
refined tools which were developed in the theory of laminates and differential inclusions and they
present some substantial differences with the proofs in \cite{DS1,DS2}. In order
to address these differences we start by recalling some of the standard  notions in the
theory of differential inclusions. These notions have been developed in the particular case of
gradient vector fields (i.e. when \eqref{e:LRbis} is given by ${\rm curl}\, z = 0$). Nevertheless the
extension to systems of the form \eqref{e:LRbis}-\eqref{e:CRbis} is straightforward.

\begin{remark} In the case of Theorem \ref{t:Shvid} several modifications of this general strategy
are necessary. First of all, as mentioned above, \eqref{e:LRbis} needs to be replaced by \eqref{e:ac_lr}.
This in turn means that in general it is not possible to obtain a sequence $\{z_j\}$ as in (P) if
we insist that each function is compactly supported. It is however possible to build up such a sequence
if instead of requiring that its support be compact, we require that $z_j$ converges uniformly to $0$ in
the complement of $B_1 (0)$ (cp. Lemma 2.1 in \cite{Shvydkoy}).
\end{remark}

Returning to the general considerations at the beginning of this chapter, we recall that a central point is to find an
open set $\mathcal{U}$ satisfying the perturbation property (P). 
One possible candidate would be to take the largest open set $\mathcal{U}_{max}$ satisfying (P). Obviously this set is particularly meaningful since it gives the largest possible space $X$ for which the genericity conclusion above holds. 
Moreover, this has the advantage that -- at least in 
many relevant cases -- the set $\mathcal{U}_{max}$ coincides with the interior of the lambda-convex hull $K^{\Lambda}$, which in turn can be characterized by separation arguments. For instance, in Theorem \ref{t:criterion} condition \eqref{e:ass2} characterizes precisely the interior of $K^\Lambda$. Moreover,
in this case of the interior of $K^\Lambda$ is the interior of the convex hull $K^{co}$. 
As a side remark observe that, by Jensen's inequality, any set $\mathcal{U}$ with property (P) must
necessarily be a subset
of the convex hull of $K$.

In \cite{CFG} and \cite{Shvydkoy}
the authors avoid calculating the full hull and instead restrict themselves to exhibiting a 
nontrivial (but possibly much smaller) open set $\mathcal{U}$ satisfying (P). However,
in exchange they are forced to use much more complicated sequences $z_j$. Indeed, 
the $z_j$'s used in \cite{DS1} are localizations of simple
plane waves, whereas the ones used in \cite{CFG} and \cite{Shvydkoy} arise as an infinite
nested sequence of repeated plane waves.

The obvious advantage of the method introduced in \cite{CFG} and used in \cite{Shvydkoy}
is that it seems to be fairly robust and general. This is useful in cases where 
an explicit computation of the hull $K^{\Lambda}$ (or even of the convex hull!) is out of reach due to the
high complexity and high dimensionality.

On the other hand, the advantage of computing the hull $K^\Lambda$ is that it gives optimal criteria for
wild initial data, as in Section \ref{s:wild}. Indeed, observe that in Theorem \ref{t:criterion},
the initial data $v_0$ is given by the subsolution through \eqref{e:initialdatum}. In the case of 
Theorem \ref{t:CFG}, that is, for the incompressible porous medium equation, the set $K^\Lambda$
was computed in \cite{Szekelyhidi}. As a consequence one obtains the following existence
theorem for the Muskat problem (see Theorem 1.1 in \cite{Szekelyhidi}).

\begin{theorem}\label{t:lsz-ipm}
Assume $m$ is given by \eqref{e:IPM}. There exist infinitely many weak solutions of 
\eqref{e:ac_lr} and \eqref{e:ac_nc} in $L^\infty(\R^2 \times [0, +\infty))$ with $|\theta|=1$ a.e. and 
\[
\theta_0 (x)\;=\; \left\{
\begin{array}{ll}
1 & \mbox{if $x_2>0$}\\
-1 & \mbox{if $x_2<0$}\, .
\end{array}\right.
\]
\end{theorem}

In fact the solutions constructed in this theorem exhibit a mixing zone (i.e. where $\theta$ oscillates wildly between $\pm 1$) 
around the initial interface $\{x_2=0\}$, which is expanding linearly in time. 

We recall that the coarse-grained flow for this problem has been analysed in detail in\cite{Otto1,Otto2}. 
In \cite{Otto1} F.~Otto introduced a relaxation approach for the incompressible porous media equation, based on 
the gradient flow formulation and using ideas from mass transport. It was shown that, under certain assumptions, there exists a {\it unique} ''relaxed'' solution $\overline{\theta}$, representing a kind of coarse-grained density. Moreover, Otto proved  that, in general, the mixing zone (where the coarse-grained density $\overline{\theta}$ is strictly between $\pm 1$) grows linearly in time, with the possible exception of a small set of volume fraction $O(t^{-1/2})$ (cp. with \cite[Remark 2.1]{Otto2}). Following the general considerations in Section \ref{ss:reynolds} we can interpret $\overline{\theta}$ as the subsolution corresponding to the solutions in Theorem \ref{t:lsz-ipm}, for which
there exists a sequence $\rho_k$ of weak solutions such that
$\rho_k\overset{*}{\rightharpoonup} \overline{\rho}$. 

Although subsolutions are clearly not unique, the coarse-grained density of Otto turns out to be extremal in the sense that it corresponds to the maximal expansion of the mixing zone. It is interesting to note in this connection, that, although weak solutions are clearly not unique, there is a way to identify a selection criterion among subsolutions which leads to uniqueness.

%%%%%%%%%%%%%%%%%%%%%%%%%%%%%%%
\section{Gromov's h-principle and the Nash-Kuiper Theorem} 
%%%%%%%%%%%%%%%%%%%%%%%%%%%%%%%

The origin of convex integration lies in the famous Nash-Kuiper theorem. 
In this section we briefly recall some landmark results from the theory of isometric embeddings.

Let $M^n$ be a smooth compact manifold 
of dimension $n\geq 2$, equipped with a Riemannian metric $g$. An isometric immersion
of $(M^n,g)$ into $\R^m$ is a map $u\in C^1(M^n;\R^m)$ such that the induced metric agrees with $g$. In local coordinates this amounts to 
the system
\begin{equation}\label{e:equations}
\partial_iu\cdot\partial_ju=g_{ij}
\end{equation}
consisting of $n(n+1)/2$ equations in $m$ unknowns. If in addition $u$ is injective, it is an isometric embedding. 
Assume for the moment that $g\in C^{\infty}$. The two classical theorems concerning the solvability of this system are: 
\begin{enumerate}
\item[(A)] if $m\geq (n+2)(n+3)/2$, then any short embedding can be uniformly approximated by  isometric embeddings of class $C^{\infty}$ (Nash \cite{Nash56}, Gromov \cite{Gromov});
\item[(B)] if $m\geq n+1$, then any short embedding can be uniformly approximated by isometric embeddings of class $C^1$ (Nash \cite{Nash54}, Kuiper \cite{Kuiper55}).
\end{enumerate}
Recall that a short embedding is an injective map $u:M^n\to \R^m$ such that the metric induced on $M$ by $u$ is shorter than $g$. In coordinates this means that
\begin{equation}\label{e:sub_for_nash}
(\partial_iu\cdot\partial_ju)\leq (g_{ij})
\end{equation}
in the sense of quadratic forms. Thus, (A) and (B) are not merely existence theorems, they show that there exists a huge (essentially $C^0$-dense) set of solutions. This type of abundance of solutions is a central aspect of Gromov's $h$-principle, for which the isometric embedding problem is a primary example (see \cite{Gromov,Eliashberg}). 

There is a clear formal analogy between \eqref{e:equations}-\eqref{e:sub_for_nash} and \eqref{e:Cauchy},\eqref{e:LR},\eqref{e:CR}.
First of all, note that the Reynold stress measures the defect to being a solution of the Euler equations and it is in general a nonnegative
symmetric tensor, whereas $g_{ij} - \partial_iu \cdot \partial_j u$ measures the defect to being isometric and, for a short
map, is also a nonnegative symmetric tensor.
More precisely \eqref{e:equations} can be formulated for the deformation gradient $A:= Du$ as 
the coupling of the linear constraint 
$$
{\rm curl}\, A = 0
$$
with the nonlinear relation 
$$
A^t A = g.
$$
In this sense short maps are "subsolutions" to the 
isometric embedding problems in the spirit of Definition \ref{d:subsolution}. Along this line of thought, Theorem \ref{t:criterion} is then the analogue
for the Euler equations of the Nash-Kuiper result (B). However note that, strictly speaking, the formal analog of statement (B) would
be replacing $L^\infty$ by $C^0$ in Theorem \ref{t:criterion}. This type of result is not yet available.

\medskip

Statement (B) is rather surprising for two reasons. First of all, for $n\geq 3$ and $m=n+1$, the system \eqref{e:equations}
is overdetermined. Moreover, for $n=2$ we can compare (B) to the classical rigidity result concerning the Weyl problem: if $(S^2,g)$ is a 
compact Riemannian surface with positive Gauss curvature and $u\in C^2$ is an isometric immersion into $\R^3$, then $u$ is uniquely determined up to a rigid motion (\cite{CohnVossen,Herglotz}, see also \cite{Spivak5} Chapter 12 for a thorough discussion). Thus it is clear 
that isometric immersions have a completely different qualitative behaviour  at low and high regularity (i.e. below and above $C^2$).

A strikingly similar phenomenon holds for the Euler equations: the when coupled with the energy
constraint $|v|^2 = 2\bar{e}$ they are also formally overdetermined. Moreover $C^1$ solutions of the Cauchy problem are unique.
We will return to these similarities in Section \ref{s:C1alpha}.

%%%%%%%%%%%%%%%%%%%%%%%%%%%%%%%
\section{Onsager's Conjecture}
%%%%%%%%%%%%%%%%%%%%%%%%%%%%%%%

The relevance of weak solutions of the Euler equations to turbulence has long been surmised. One key question is related to the phenomenon of anomalous dissipation. This experimentally observed fact, namely that the rate of energy dissipation in the vanishing viscosity limit stays above a certain non-zero constant, is expected to  arise from a mechanism of transporting energy from large to small scales (known as an energy cascade) via the nonlinear transport term in the Navier-Stokes equations, rather than the (dissipative) viscosity term. Motivated by this idea, Onsager conjectured in the 1940s that there exist weak solutions of the Euler equations which dissipate energy. More precisely, the conjecture was that there exist dissipative weak solutions in the H\"older space $C^{\alpha}$ with $\alpha<1/3$, whereas if $\alpha>1/3$, then the energy is conserved \cite{Onsager,EyinkOnsager}. 

The latter has been rigorously proved with various sharper versions in \cite{Eyink,ConstantinETiti,DuchonRobert}, in particular we refer the reader to the recent survey \cite{ShvydkoyOnsager}. 

On the other hand there is no known construction of a dissipative weak solution in $C^{\alpha}$ with any $\alpha\geq 0$. As pointed out in Theorem \ref{t:energy_sucks}, one can use Theorem \ref{t:criterion} to produce dissipative solutions which are merely $L^{\infty}$, and prior to that the only known construction of a weak solution with dissipation, due to Shnirelman \cite{Shnirelmandecrease} produced an $L^2$ solution. Despite this, it is instructive to take a second look at the construction in Theorem \ref{t:criterion} in light of expectations regarding the energy spectrum and the conjecture of Onsager.

In Section \ref{s:applications} we presented the so called Baire category method, which is in some sense not constructive. However,
the same idea of adding oscillatory perturbations can be implemented in a constructive way as well. See for instance Section 5 of \cite{DS1}. In a nutshell the idea is to define a sequence of subsolutions $(v_k,u_k,q_k)$ as
\begin{equation}\label{e:constructive}
(v_{k+1},u_{k+1},q_{k+1})(x,t)=(v_k,u_k,q_k)(x,t)+(V_k,U_k,Q_k)(x,t,\lambda_kx,\lambda_kt),
\end{equation}
where 
$$
(V_k,U_k,Q_k)(x,t,\xi,\tau)
$$
is a periodic plane-wave solution of \eqref{e:LR} in the variables $(\xi,\tau)$ with average 0, parametrized by $(x,t)$, and $\lambda_k$ is a (large) frequency to be chosen. The aim is to choose the plane-wave $(V_k,U_k,Q_k)$ and the frequency $\lambda_k$ iteratively in such a way that 
\begin{itemize}
\item $(v_k,u_k,q_k)$ continues to satisfy \eqref{e:LR} (strictly speaking this requires an additional corrector term in the scheme \eqref{e:constructive});
\item the inequality \eqref{e:CR} holds
\item and
\begin{equation}\label{e:terza}
v_k\otimes v_k-u_k\to \frac{2}{n}\bar{e}I.
\end{equation}
\end{itemize}
Observe that, because of the inequality \eqref{e:CR}, it suffices to show the weak convergence in
\eqref{e:terza}.

The role of the frequency can be explained as follows: using \eqref{e:constructive} we expand as
$$
v_{k+1}\otimes v_{k+1}-u_{k+1}=v_k\otimes v_k-u_k+\biggl(v_k\otimes V_k+V_k\otimes v_k-U_k\biggr)+V_k\otimes V_k.
$$
For large $\lambda_k$ the term in brackets converges to zero weakly, whereas the term $V_k\otimes V_k$ is non-negative. 
Therefore, if the sequence $\{\lambda_k\}_{k\in \N}$ converges to $\infty$ sufficiently fast, then the sequence of tensors $\{v_k\otimes v_k-u_k\}$ is monotone, with a uniform bound given by \eqref{e:CR}. The strong convergence follows.

Thus, convergence of this constructive scheme is improved by choosing the frequencies $\lambda_k$ higher and higher. In terms of the energy spectrum this amounts to eddies (the periodic perturbations) being placed in well separated frequencies where the interaction between neighboring frequencies is negligible.

On the other hand clearly any (fractional) derivative or H\"older norm of $v_k$ gets worse by such a choice of $\lambda_k$. The best regularity admitted by the scheme corresponds to the slowest rate at which the frequencies $\lambda_k$ tend to infinity whilst still leading to convergence. 

%%%%%%%%%%%%%%%%%%%%%%%%%%%%%%%
\subsection{A toy example}\hfill
%%%%%%%%%%%%%%%%%%%%%%%%%%%%%%%

We will consider here a toy model, where the considerations above can be demonstrated and which has been suggested by the anonymous referee to \cite{DS1}. Consider the problem of exhibiting functions $u:[0,1]\to \R$ such that
$|u|=1$. In the context of \eqref{e:LRbis}-\eqref{e:CRbis} this corresponds to $K=\{-1,1\}$ and the differential constraint being void. The following scheme aims at producing such functions. We assume to start with
a given function $u_0:[0,1]\to \R$ and build a sequence with the following iteration scheme:
$$
u_{k+1}(x)=u_k(x)+\tfrac{1}{2}\bigl[1-u_k^2(x)\bigr]s(\lambda_kx)\, ,
$$
where $s:\R\to\R$ is the 1-periodic extension of ${\bf 1}_{(0,1/2]}-{\bf 1}_{(1/2,1]}$ and $\lambda_k>1$ is a sequence of frequencies still to be fixed. The following assertions are straightforward:
\begin{itemize}
\item If $\sup_{[0,1]}|u_k|<1$, then also $\sup_{[0,1]}|u_{k+1}|<1$. 
\item If $\sup_{[0,1]}|u_0|<1$ and $u_k\to u_{\infty}$ in $L^1(0,1)$, then $|u_{\infty}|=1$ a.e.
\end{itemize}
Therefore, in order to produce a solution to our toy problem, it suffices to choose the sequence
$\{\lambda_k\}$ so to ensure the strong convergence of $u_k$. To this end observe that
$$
\int_0^1|u_{k+1}|^2\,dx=\int_0^1\left(|u_k|^2+\tfrac{1}{4}(1-u_k^2)^2+u_k(1-u_k^2)s(\lambda_kx)\right)\,dx.
$$
Moreover, as $\lambda\to\infty$, we have $s(\lambda x)\rightharpoonup 0$ in $L^2(0,1)$. Therefore, by choosing $\lambda_k$ sufficiently large (depending on $u_k$), we can ensure that
$$
\int_0^1|u_{k+1}|^2\,dx\geq \int_0^1|u_k|^2\,dx+\int_0^1 \tfrac{1}{8}(1-u_k^2)^2\,dx.
$$
The strong convergence follows then easily. Here we see that choosing $\lambda_k$ to be a rapidly increasing sequence ``helps'' the strong convergence of the scheme.  

However, it is also clear that for any additional regularity of the limit $u_{\infty}$ one should choose $\lambda_k$ to increase as slowly as possible. More precisely, the optimal regularity that is reachable via this iteration scheme will depend on the connection between the choice of $\lambda_k$ with the rate of convergence of the scheme. To see this, observe that -- roughly speaking -- fractional Sobolev regularity of $u_{\infty}$ will follow from interpolating between the norms
\begin{eqnarray*}
\|u_{k+1}-u_k\|_{L^1}&\sim& \int_0^1(1-u_k^2)\,dx\\
\|u_{k+1}-u_k\|_{BV}&\sim& \lambda_k\int_0^1(1-u_k^2)\,dx
\end{eqnarray*}
Therefore the following statement is of interest, showing that exponential growth of the frequencies leads to exponential convergence of the scheme:

\begin{lemma}
Let $\lambda_k=2^k$. Then 
$$
\int_0^1(1-u_k^2)\,dx\leq \left(\tfrac{7}{8}\right)^k\int_0^1(1-u_0^2)\,dx.
$$
\end{lemma}

\begin{proof}
By the choice of the oscillatory function $s$ we see that for any continuous function $f\in C(-1,1)$ we have
$$
\int_0^1f(u_{k+1})\,dx=\int_0^1\tfrac{1}{2}\bigl[f\bigl(u_k-\tfrac{1}{2}(1-u_k^2)\bigr)+ \tfrac{1}{2}f\bigl(u_k+\tfrac{1}{2}(1-u_k^2)\bigr)\bigr]\,dx
$$
Now set $f(u):=(1-u^2)^{1/2}$. By direct calculation we obtain
$$
f''(u)=-(1-u^2)^{-3/2}
$$
and $\tfrac{d^4f}{du^4}\leq 0$ on $(-1,1)$. It follows that $f''(u)(1-u^2)^2=-f(u)$ and
$$
\tfrac{1}{2}f(u-v)+\tfrac{1}{2}f(u+v)\leq f(u)+\tfrac{1}{2}f''(u)v^2.
$$
In particular, setting $v=\tfrac{1}{2}(1-u^2)$ we obtain
$$
\tfrac{1}{2}f\Bigl(u-\tfrac{1}{2}(1-u^2)\Bigr)+ \tfrac{1}{2}f\Bigl(u+\tfrac{1}{2}(1-u^2)\Bigr)\leq \frac{7}{8}f(u).
$$
We conclude
$$
\int_0^1f(u_{k+1})\,dx\leq \frac{7}{8}\int_0^1f(u_k)\,dx
$$
and the lemma follows.
\end{proof}

%%%%%%%%%%%%%%%%%%%%%%%%%%%%%%%
\subsection{ $C^{1,\alpha}$ isometric immersions}\label{s:C1alpha}\label{s:c1alpha}\hfill
%%%%%%%%%%%%%%%%%%%%%%%%%%%%%%%

The question of a sharp regularity threshold has been the object of investigation for the isometric embedding of surfaces as well
(see for instance \cite{Gromov}, \cite{Yau}). As already mentioned, the isometric embeddings of ${\mathbb S}^2$ into $\R^3$ are rigid in the class $C^2$, whereas the h-principle holds for $C^1$. Borisov investigated embeddings of class $C^{1,\alpha}$, and proved the rigidity for $\alpha>2/3$ (\cite{BorisovRigidity1}, \cite{BorisovRigidity2}) and the h-principle for $\alpha<1/13$ (although the latter was announced in 1965, see\cite{Borisov65}, a partial proof only appeared in 2004 \cite{Borisov2004}). In \cite{CDSz} we returned to this problem, and gave a more modern PDE proof of the h-principle for $\alpha<1/7$, namely

\begin{theorem}[Local existence]\label{t:local}
Let $n\in \N$ and $g_0\in \sym^+_n$. There
exists $r>0$ such that the following holds for
any smooth bounded open set $\Omega\subset\R^n$
and any Riemannian metric $g\in C^{\beta}(\overline{\Omega})$ with 
$\beta>0$ and $\|g-g_0\|_{C^0}\leq r$. 
There exists a constant $\delta_0>0$ such that, if 
$u\in C^2(\overline{\Omega};\R^{n+1})$ and $\alpha$ satisfy
$$
\|u^\sharp e-g\|_0\leq \delta_0^2 \qquad \mbox{and}\qquad
0<\alpha<\min\left\{\frac{1}{1+2n_*},\frac{\beta}{2}\right\}\, ,
$$ 
then there exists a map $v\in C^{1,\alpha} (\overline{\Omega}; \R^{n+1})$ with 
$$
v^\sharp e \;=\; g\qquad\mbox{and}\qquad
\|v-u\|_{C^1} \;\leq\; C\, \|u^\sharp e-g\|_{C^0}^{1/2}\, .
$$
\end{theorem}

\begin{corollary}[Local h-principle]\label{c:local}
Let $n,g_0,\Omega,g,\alpha$ be as in Theorem \ref{t:local}. Given any
short map $u\in C^1(\overline{\Omega};\R^{n+1})$ and any $\e>0$ there exists an isometric immersion $v\in C^{1,\alpha}(\overline{\Omega};\R^{n+1})$ with $\|u-v\|_{C^0}\leq \e$.
\end{corollary}

The proof of Theorem \ref{t:local} is based on an iteration scheme which follows the method of Nash and Kuiper \cite{Nash54,Kuiper55}. 
The iteration consists of {\it stages}, and each {\it stage} consists of several {\it steps}. The purpose of a {\it stage} is 
to correct the error $g-u^\sharp e$. In order to achieve this correction, the error is decomposed into a sum of primitive metrics as
\begin{equation}\label{e:decomp}
g-u^\sharp e = \sum_{k=1}^{n_*}a_k^2\nu_k\otimes\nu_k\, .
\end{equation}
To keep the notation simpler, in what follows we will use the abbreviations
$\|\cdot \|_N$ and $\| \cdot \|_{k,\alpha}$ for the $C^N$ and $C^{k, \alpha}$ norms respectively.
The natural estimates associated with \eqref{e:decomp} are
\begin{equation*}
\begin{split}
\|&a_k\|_0\;\sim\;\|g-u^\sharp e\|_0^{1/2}\\
\|&a_k\|_{N+1}\;\sim\;\,\|u\|_{N+2}\quad\textrm{ for }N=0,1,2,\dots.
\end{split}
\end{equation*}
A {\it step} then involves adding one primitive metric. In other words the goal of a {\it step} is the metric change
$$
u^\sharp e\quad\mapsto\quad u^\sharp e+a^2\nu\otimes\nu.
$$
Nash used spiralling perturbations (also known as the "Nash twist") to achieve this; in codimension one Kuiper replaced the spirals by corrugations: in both cases the new map $\tilde{u}$ is of the form 
\[\tilde{u} (x) = u(x) + \frac{1}{\lambda} w (x, \lambda x\cdot \nu)
\]
for some appropriate choice of $w$, where $w$ is periodic in the second variable.
The parameter $\lambda$ is given by 
\[
\lambda = K \frac{\|u\|_2}{\sqrt{\|g-u^\sharp e\|_0}}
\]
where $K$ is a large (but fixed) constant.
One checks that the addition of a primitive metric is then 
possible with the following estimates:
\begin{eqnarray*}
\textrm{$C^0$-error in the metric}\;&\sim&\;\|g-u^\sharp e\|_0\,\frac{1}{K}\\
\textrm{ increase of $C^1$-norm of $u$}\;&\sim&\;\|g-u^\sharp e\|_0^{1/2}\\
\textrm{ increase of $C^2$-norm of $u$}\;&\sim&\;\|u\|_2\,K
\end{eqnarray*}
for any $K\geq 1$. Observe that the first two of these estimates is essentially the same as in \cite{Nash54,Kuiper55}. Furthermore, the third estimate is only valid modulo a "loss of derivative".  The exponent $1/2$ in the second estimate is due to the quadratic nature
of the nonlinearity.

The low codimension forces the steps to be performed serially. Thus the number of {\it steps} in a {\it stage} equals the number of primitive metrics in the above decomposition which interact. This is given by $n_* = n (n+1)/2$. 
To deal with the "loss of derivative" problem in \cite{CDSz} we mollify the map $u$ at the start of every stage, in a similar manner as is done in a Nash-Moser iteration (see \cite{Moser66}, \cite{Moser66b}). A careful estimate (see  Proposition \ref{p:improv}) shows that this can be done without any 
additional error. 
Therefore, iterating the estimates for one step over a single stage (that is, over $n_*$ steps) leads to 
\begin{eqnarray}
\textrm{$C^0$-error in the metric}\;&\sim&\;\|g-u^\sharp e\|_0\,\frac{1}{K}\label{e:stage1}\\
\textrm{ increase of $C^1$-norm of $u$}\;&\sim&\;\|g-u^\sharp e\|_0^{1/2}\label{e:stage2}\\
\textrm{ increase of $C^2$-norm of $u$}\;&\sim&\;\|u\|_2\,K^{n_*}\label{e:stage3}
\end{eqnarray}
With these estimates, iterating over the {\it stages} leads to exponential convergence of the metric error, to a controlled growth of the $C^1$ norm and to an exponential growth of the $C^2$ norm of the map. In particular, interpolation between these two norms leads to convergence in $C^{1,\alpha}$ for $\alpha<\frac{1}{1+2n_*} = \frac{1}{1+n(n+1)}$.

Recalling the analogy with the Euler equations, which amounts to identifing the velocity $v$ with the deformation gradient $Du$, 
the analogous scheme for producing $C^{0,\alpha}$ solutions of Euler would involve perturbations of the form
\[
\tilde{v} (x,t) = v (x,t) + w (x,t, \lambda x,\lambda t)
\]
(with $w$ periodic and average $0$ in the second variable). If the estimates analogous to \eqref{e:stage1}--\eqref{e:stage3}
would carry over with one single step in a stage (i.e. $n_*=1$), one would reach the exponent conjectured by Onsager.
This ansatz for the critical exponent $1/3$ is independent of considerations on the energy spectrum, giving 
an entirely new point of view on the famous Kolmogorov $5/3$ law.

%%%%%%%%%%%%%%%%%%%%%%%%%%%%%%%
\subsection{Conservation of energy and rigidity}\label{s:rigidity}\hfill
%%%%%%%%%%%%%%%%%%%%%%%%%%%%%%%

We turn next to the Weyl problem, i.e. the characterization of isometric embeddings of
$2$-dimensional positively curved surfaces in $\R^3$. Concerning the rigidity, it is known from the work of 
Pogorelov and Sabitov that 
\begin{enumerate}
\item closed $C^1$ surfaces with positive Gauss curvature 
and bounded extrinsic curvature are convex (see \cite{Pogorelov73});
\item closed convex surfaces are rigid in the sense that isometric 
immersions are unique up to rigid motion \cite{PogorelovRigidity};
\item a convex surface with metric $g\in C^{l,\beta}$ with $l\geq 2,0<\beta<1$ and positive curvature is of class $C^{l,\beta}$ (see \cite{Pogorelov73,Sabitov}).
\end{enumerate}
Thus, extending the rigidity in the Weyl problem to $C^{1,\alpha}$ isometric immersions can be reduced to showing that the image of the surface has bounded extrinsic curvature in the sense of Pogorelov. This concept is easily explained in the following way.
Let $(M,g)$ be a smooth $2$-dimensional closed Riemannian manifold and $u: M\to \R^3$ an isometric immersion.
According to Gauss' original definition, the Gauss curvature of $u(M)$ can be defined as the area distortion of the 
Gauss map $N$. In a weak sense, this is a well-defined object even when the immersion is merely a $C^1$ map, because
in this case $N$ is continuous. The immersed surface $u(M)$ is
then said to have bounded extrinsic curvature if this area distortion is bounded. 

Using geometric arguments, in a series of papers \cite{Borisov58-1,Borisov58-2,BorisovRigidity1,Borisov58-3,BorisovRigidity2} Borisov proved that for $\alpha>2/3$ the image of surfaces with positive Gauss curvature has indeed bounded extrinsic curvature. Consequently, rigidity holds in this range and in particular $2/3$ is an upper bound on the range of H\"older exponents that can be reached using convex integration. 

A short proof of Borisov's rigidity result was provided in \cite{CDSz}. Note that if $u\in C^3$ one can compute
the area distorion of the Gauss map from the Riemann-curvature tensor, which in turn depends only on the metric:
we call this scalar field $\kappa$. Even if
the metric $g$ is smooth, nonetheless this identity is in general false if the isometry is not regular enough, as shown precisely
by the Nash-Kuiper theorem. However, in \cite{CDSz} we showed that the $C^{1,\frac{2}{3} +\eps}$ regularity is enough for this identity to hold.
A key observation of the proof in \cite{CDSz} is that this identity can be expressed in the following integral
formula
\begin{equation}\label{e:ChangeOfVar}
\int_V f(N(x)) \kappa (x)\, dA(x)\;=\;
\int_{\S^2}f(y)\deg(y,V,N)\, d\sigma(y)
\end{equation}
where
\begin{itemize}
\item $V$ is an arbitrary open subset of $M$;
\item $f$ is any $C^1$ function on ${\mathbb S}^2$;
\item $\deg (y, V, N)$ is the Brouwer
degree of the map $N|_V$ at $y$.
\end{itemize}

The identity \eqref{e:ChangeOfVar} is just a change of variable formula and, surprisingly, it can be proved for 
immersion $u\in C^{1, 2/3+\eps}$ by a regularization argument. The corresponding computations in \cite{CDSz} rely on commutator
estimates which are surprisingly close to the argument of \cite{ConstantinETiti} for proving the energy 
conservation of $C^{0,\alpha}$ solutions of Euler when $\alpha>1/3$. 
In the case of isometric embeddings there does
not seem to be a universally accepted critical exponent (cp. Problem 27 of \cite{Yau}), even though 1/2 and 1/3 both
seem relevant (cp. with our discussion below and with the discussion in \cite{Borisov2004}).
The key commutator estimates of \cite{CDSz} are given by the following proposition, which can be
compared to \cite{ConstantinETiti}.

\begin{proposition}
[Quadratic estimate]\label{p:improv}
Let $\Omega\subset \R^n$ be an open set, 
$v\in C^{1, \alpha} (\Omega, \R^m)$ and $e$ the Euclidean metric on $\R^m$.
Consider the corresponding pull-back metric
$(v^\sharp e )_{ij} = \partial_i v \cdot \partial_j v$. Assume $v^\sharp e\in C^2$ and let
$\varphi\in C^\infty_c (\R^n)$ a standard symmetric
convolution kernel. Then, for every compact set 
$K\subset \Omega$, 
\begin{equation}\label{e:improv}
\|(v*\varphi_\ell)^\sharp e - v^\sharp e\|_{C^{1} (K)} \;=\;  
O (\ell^{2\alpha -1}).
\end{equation}
\end{proposition}

In particular, fix a map $u$ and a kernel
$\varphi$ satisfying the assumptions of
the Proposition with $\alpha>1/2$.
Then the Christoffel symbols of $(v*\varphi_\ell)^\sharp e$
converge {\em uniformly} to those of $v^\sharp e$. This corresponds to the results of Borisov in
\cite{Borisov58-1,Borisov58-2}, and hints at the
absence of $h$--principle for $C^{1,\frac{1}{2}+\eps}$ immersions.
One might further notice that the regularity $C^{1,\frac{1}{3}+\eps}$ is still enough
to guarantee a very weak notion of convergence of the corresponding Christoffel symbols.

%%%%%%%%%%%%%%%%%%%%%%%%%%%%%%%
\section{Outlook and further problems}
%%%%%%%%%%%%%%%%%%%%%%%%%%%%%%%

We list in this section several open questions which remain to be addressed.

\begin{problem}
Analogues of Theorem \ref{t:shear} for compressible Euler.
\end{problem}
Different behaviours are expected for different types of Riemann problems. The direct
analogue of Theorem \ref{t:shear} would be a contact discontinuity, and it seems likely that
some form of the statement survives for compressible Euler. Nevertheless, because of the
different role of the pressure term, a direct transfer from the incompressible to the compressible case (as
was done for Theorem \ref{t:psistema} in \cite{DS2}) seems impossible. 

\begin{problem}
Compute the hull for compressible Euler.
\end{problem}
The proof of Theorem \ref{t:psistema} essentially relies on the observation that
Theorem \ref{t:criterion} gives weak solutions of the incompressible Euler system \eqref{e:Cauchy} 
with fixed $|v|^2$ and fixed pressure $p$. Thus, an appropriate choice automatically yields
weak solutions of the compressible Euler system \eqref{e:psistema}.  Within the general framework of 
\eqref{e:LRbis}-\eqref{e:CRbis} this route via the incompressible
Euler equations amounts to restricting to a smaller constitutive set $K$ for compressible Euler, where $\rho$ is fixed to be constant.
A further step was obtained recently by Chiodaroli in Theorem \ref{t:chiodaroli}, but the general case, which would allow for wild oscillations in the density as well as the velocity, remains open. In other words, one should define the state space and the constitutive set for compressible Euler in full generality, and compute the hull. This would lead to the correct definition of subsolution, and thus eventually to the solution of the next problem:

\begin{problem}
Analogue of Theorem \ref{t:mv} for compressible Euler.
\end{problem}

As noted in Section \ref{s:rigidity}, although there is a surprising number of similarities between
weak solutions of the Euler equations and the Nash-Kuiper theorem, in contrast to Euler there are no natural conjectures 
for critical H\"older exponents for isometric embeddings. Therefore, a central question is: 

\begin{problem}
Find the optimal H\"older-exponent for Theorem \ref{t:local} and Corollary \ref{c:local}.
\end{problem}

It should be pointed out that for incompressible Euler the exponent $1/3$ is conjectured to be critical with respect to the 
\emph{energy conservation}. It is not known whether being above this exponent implies uniqueness. On the other hand 
the rigidity statement in Section \ref{s:rigidity} is a uniqueness statement. A conceivable
scenario is then the existence of other regimes between the exponents $1/3$ and $2/3$ for isometric embeddings,
where the rigidity does not hold but the $h$-principle is also excluded.

In Section \ref{s:C1alpha} we restricted our attention to isometric embeddings in codimension 1. Looking closely at the construction it
is natural to expect that higher codimension for the embedding should result in better regularity. Indeed, for the case of sufficiently large codimension A.~K\"allen showed in \cite{Kallen} that the regularity in Theorem \ref{t:local} and Corollary \ref{c:local} can be pushed up to any $\alpha<1$, provided the metric is sufficiently smooth!
 
\begin{problem}
Analyze the effect of higher codimension on the regularity in Theorem \ref{t:local} and Corollary \ref{c:local}.
\end{problem}
In other words, interpolate between Theorem \ref{t:local} (codimension 1) and K\"allen's result (codimension sufficiently large).

\bigskip

Regarding Onsager's conjecture the next important question is how to modify convex integration so that 
the weak solutions in Theorem \ref{t:criterion} are H\"older continuous.
\begin{problem}
Construct continuous and H\"older-continuous solutions of incompressible Euler.
\end{problem}
An explicit construction could lead to insights concerning several assumptions usually made in the statistical theory of turbulence (such as
the localness of interactions between eddies, cp. \cite{CCFS}, or the very idea of an energy cascade). It would also open the way towards
proving a true homotopy-principle for weak solutions, in the long run perhaps shedding light on the global topology of the moduli space of solutions.

\bigskip

\begin{problem}
Calculate/estimate the maximal dissipation rate for wild initial data. 
\end{problem}
By a simple contradiction argument it can be seen that there has to be a finite maximal rate. 
A first example might be the shear flow, cp. with Theorem \ref{t:shear}. For the very similar problem
of the incompressible porous medium equation, see Theorem \ref{t:lsz-ipm}, the maximal expansion rate of
the mixing zone has been obtained in \cite{Otto2}. 

Apart from a dissipation rate principle, the natural candidate for a selection principle for 
the incompressible Euler equations is the vanishing viscosity limit. Of course this has several major technical stumbling blocks, but 
already a ``recovery sequence'' in the vanishing viscosity limit would be of interest:
\begin{problem}
Recovery sequence for weak solutions of Euler from Navier-Stokes.
\end{problem}
It is conceivable, that one would need to consider Navier-Stokes with a random perturbation. 
Indeed, adding noise is not just physically motivated, but in some sense necessary: 
without noise the statement is false (see Theorem \ref{t:shear}).

\begin{problem}
Probabilistic approach to convex integration. 
\end{problem}
Convex integration can be seen as a control problem: at each step of the iteration, one has to choose
an admissible perturbation, consisting essentially of a (plane-)wave direction and a frequency. The following
is therefore a natural question: are there meaningful optimality conditions?

Finally, up to now there has been no attempt to link weak solutions and/or subsolutions to generalized flows or to Brenier's sharp measure-valued solution \cite{Brenier1999}:
\begin{problem}
Relation of subsolutions to generalized flows. 
\end{problem}

%\bibliographystyle{acm}
%\bibliography{eulerbib}

\end{document}